\documentclass[12pt, reqno]{amsart}
\usepackage{amsmath, amsthm, amscd, amsfonts, amssymb, graphicx, color}
\usepackage[bookmarksnumbered, colorlinks, plainpages]{hyperref}
\hypersetup{colorlinks=true,linkcolor=red, anchorcolor=green, citecolor=cyan, urlcolor=red, filecolor=magenta, pdftoolbar=true}
\newtheoremstyle{uprightplain}
{6pt}{6pt}
{\normalfont}
{}{\bfseries}{.}{ }{}
\theoremstyle{uprightplain}
\newtheorem{theorem}{Theorem}[section]
\newtheorem{lemma}[theorem]{Lemma}
\newtheorem{proposition}[theorem]{Proposition}
\newtheorem{cor}[theorem]{Corollary}
\theoremstyle{definition}

\newtheorem{question}[theorem]{Question}
\newtheorem{example}[theorem]{Example}
\newtheorem{remark}[theorem]{Remark}
\numberwithin{equation}{section}
\allowdisplaybreaks
\begin{document}
	
	\title [Weighted composition operators]{\Small{Weighted composition operators on weighted Dirichlet spaces: boundedness, compactness and spectral properties}}
	\author[A. Sen]{Anirban Sen}
	
	\address[Sen] {Mathematical Institute, Silesian University in Opava, Na Rybn\'{\i}\v{c}ku 1, 74601 Opava, Czech Republic}
	\email{anirbansenfulia@gmail.com; Anirban.Sen@math.slu.cz}
	
	
	
	
	\maketitle	
	
	\begin{abstract} 
		We establish necessary and sufficient conditions for the boundedness and compactness of weighted composition operators acting on weighted Dirichlet spaces and determine the spectrum of a certain class of such operators. Our results extend earlier work on unweighted composition operators and highlight the close interplay between the operator theoretic behavior of weighted composition operators and the function theoretic properties of their inducing functions. Several examples are provided to illustrate the applicability of the obtained results.
	\end{abstract}
	
	\section{Introduction}
	Let $\text{Hol}(\mathbb{D})$ denote the space of complex-valued holomorphic functions on the open unit disc $\mathbb{D}=\{z \in \mathbb{C} : |z|<1\}.$
	For $\psi, \varphi \in \text{Hol}(\mathbb{D}),$ with $\varphi(\mathbb{D})\subseteq \mathbb{D},$ the weighted composition operator $C_{\psi,\varphi }$ on  $\text{Hol}(\mathbb{D})$ is defined by
	\[C_{\psi,\varphi }f=\psi f\circ \varphi,~~f \in \text{Hol}(\mathbb{D}).\]
	Here, $\psi$ and $\varphi$ are called the inducing functions of the operator.
	In particular, when $\psi\equiv 1,$ the operator $C_{\psi,\varphi }$ becomes the composition operator $C_{\varphi},$ and when $\varphi$ is identity map on $\mathbb D,$ the operator $C_{\psi,\varphi }$ coincides with the multiplication operator $M_{\psi}.$ 
	
	The theory of composition and weighted composition operators is well-developed, with a vast literature dedicated to understanding their behavior on different function spaces. Their boundedness and compactness on the Hardy-Hilbert space were studied in \cite{G_PAMS_2008, Shaprio_BOOK}, with corresponding results for Bergman and weighted Bergman spaces in \cite{CR_JLMS_2004, LL_JAMS_2022}. Related investigations on Besov-type spaces and the Dirichlet space appear in \cite{CGP_MA_2015, T_TAMS_2003}.

	In this article, we focus on the analysis of weighted composition operators on weighted Dirichlet spaces. For $\alpha \in \mathbb{R},$ weighted Dirichlet spaces $\mathcal{D}_{\alpha}=\mathcal{D}_{\alpha}(\mathbb{D})$ are defined by
	\[\mathcal{D}_{\alpha}=\left\{f \in \text{Hol}(\mathbb{D}) : f(z)=\sum_{n=0}^{\infty}a_nz^n, \sum_{n=0}^{\infty}(n+1)^{1-\alpha}|a_n|^2< \infty \right\}.\]

	For these spaces, necessary conditions for the boundedness and compactness of composition operators were established in \cite{MS_CJM_1986}, and a complete characterization was later provided in \cite{Z_PAMS_1998}. In contrast, weighted composition operators present a significantly more delicate situation. Although a characterization using Carleson measure was given in \cite{Kumar_2009}, these conditions are often difficult to verify explicitly.
	
	The spectral behavior of weighted composition operators is likewise strongly influenced by the inducing functions. General results on the spectral theory of composition operators can be found in \cite{AL_IJM_2004, CS_PAMS_1975, K_PJM_1979}, while results specific to Dirichlet and weighted Dirichlet spaces are discussed in \cite{CGP_MA_2015, GS_JFA_2016, H_ADM_1997}.

	In this article, we investigate how the properties of weighted composition operators on weighted Dirichlet spaces are governed by the function theoretic features of their inducing functions, with an emphasis on conditions that are easy to verify.
	Whereas the case $\alpha \geq 1$ is already well understood, we study the connection between the function theoretic properties of $\psi$ and $\varphi$ in order to characterize bounded and compact weighted composition operators $C_{\psi, \varphi}$ on $\mathcal D_{\alpha}$ for $\alpha \in (-1,1).$
	Since $C_{\psi,\varphi}$ reduces to the zero operator when $\psi \equiv 0,$ we exclude this case from consideration. By the closed graph theorem, a necessary condition for $C_{\psi,\varphi}$ to be bounded on $\mathcal D_{\alpha}$ is that $\psi$ and $\psi \varphi$ belongs to $\mathcal D_{\alpha}.$ We begin by establishing an additional necessary condition for the boundedness and compactness of weighted composition operators on $\mathcal D_{\alpha}$ for $\alpha \in (0,1),$ and demonstrate that these criteria are not sufficient for a complete characterization.
	Next, we establish sufficient conditions for the boundedness and compactness of weighted composition operators on $\mathcal D_{\alpha}$ for $\alpha \in (-1,1).$
	Our main results are presented in the following theorems.

	\begin{theorem}\label{TB1}
		Let $\varphi$ be a univalent holomorphic self-map of $\mathbb{D},$ and let $\psi \in \text{Hol}(\mathbb{D})$ satisfy
		\begin{align*}
			&(i)~~\sup_{z \in \mathbb{D}}|\psi''(z)|(1-|z|^2)<\infty,\\
			&(ii)~~\sup_{z \in \mathbb{D}}|\varphi'(z)\psi'(z)|(1-|z|^2)<\infty,\\
			&(iii)~~\sup_{z \in \mathbb{D}}|\varphi''(z)\psi(z)|(1-|z|^2)<\infty,\\
			&(iv)~~\sup_{z \in \mathbb{D}}|\varphi'(z)\psi(z)|\left(\frac{1-|z|^2}{1-|\varphi(z)|^2}\right)^{\frac{\alpha}{2}+1}<\infty,
		\end{align*}
		then $C_{\psi,\varphi }$ is bounded on $\mathcal{D}_{\alpha}.$
	\end{theorem}

	\begin{theorem}\label{TC1}
		Let $\varphi$ be a univalent holomorphic self-map of $\mathbb{D},$ and let $\psi \in \text{Hol}(\mathbb{D})$ satisfy
		\begin{align*}
			&(i)~~\lim_{|z|\to 1^-}|\psi''(z)|(1-|z|^2)=0,\\
			&(ii)~~\lim_{|z|\to 1^-}|\varphi'(z)\psi'(z)|(1-|z|^2)=0,\\
			&(iii)~~\lim_{|z|\to 1^-}|\varphi''(z)\psi(z)|(1-|z|^2)=0,\\
			&(iv)~~\lim_{|z|\to 1^-}|\varphi'(z)\psi(z)|\left(\frac{1-|z|^2}{1-|\varphi(z)|^2}\right)^{\frac{\alpha}{2}+1}=0,
		\end{align*}
		then $C_{\psi,\varphi }$ is compact on $\mathcal{D}_{\alpha}.$
	\end{theorem}
	
	We remark that these conditions are fairly simple and effective in verifying that a large class of inducing functions $\psi$ and $\varphi$ generate bounded and compact weighted composition operators, extending beyond the case of unweighted composition operators.
	
	Building on the necessary conditions obtained earlier, we present a precise characterization of when the weighted composition operator 
	$C_{\psi,\varphi}$ is bounded or compact on $\mathcal{D}_{\alpha},$ $\alpha \in (0,1)$ for certain specific choices of the inducing functions $\psi$ and $\varphi.$
	The following theorem presents this characterization.
	
	\begin{theorem}\label{mainp1}
		Let $\varphi$ be a univalent holomorphic self-map of $\mathbb{D}$ with $\varphi''\in \text{H}^{\infty}(\mathbb{D})$ and let $\psi \in \text{Hol}(\mathbb{D}).$\\
		$(i)$~~Suppose $\sup\limits_{z \in \mathbb{D}}|\psi''(z)|(1-|z|^2)<\infty.$
		Then 
		$C_{\psi,\varphi}$ is bounded on $\mathcal{D}_{\alpha}$ if and only if
		$$\sup_{z \in \mathbb{D}}|\psi(z)|\left(\frac{1-|z|^2}{1-|\varphi(z)|^2}\right)^{\frac{\alpha}{2}}<\infty.$$
		$(ii)$~~Suppose $\lim\limits_{|z|\to 1^-}|\psi''(z)|(1-|z|^2)=0.$ Then 
		$C_{\psi,\varphi}$ is compact on $\mathcal{D}_{\alpha}$ if and only if
		$$\lim_{|z|\to 1^-}|\psi(z)|\left(\frac{1-|z|^2}{1-|\varphi(z)|^2}\right)^{\frac{\alpha}{2}}=0.$$
	\end{theorem}

	As a consequence of this result, we prove a weighted version of the Comparison Theorem, which generalizes the unweighted version given in \cite{MS_CJM_1986}.

	
	Finally, under suitable function theoretic assumptions on the inducing functions $\psi$ and $\varphi$, we obtain the following complete characterization of the spectrum $\sigma(C_{\psi,\varphi};\mathcal{D}_{\alpha})$ for $\alpha \in (-1,1).$

	\begin{theorem}\label{SPMT1}
		Let $\varphi$ be a univalent holomorphic self-map of $\mathbb{D}$ with a fixed point $a \in \mathbb D,$ and satisfying $\varphi'' \in \text{H}^{\infty}(\mathbb D).$
		Let $\psi \in \text{Hol}(\mathbb{D})$ be such that:
		\begin{align*}
			&(i)~~\lim_{|z|\to 1^-}|\psi''(z)|(1-|z|^2)=0,\\
			&(ii)~~\lim_{|z|\to 1^-}|\psi(z)|\left(\frac{1-|z|^2}{1-|\varphi(z)|^2}\right)^{\frac{\alpha}{2}+1}=0.
		\end{align*}
		Then $0<|\varphi'(a)|<1$ and $\sigma(C_{\psi,\varphi};\mathcal{D}_{\alpha})=\left\{\psi(a)(\varphi'(a))^n: n \in \mathbb N \cup \{0\} \right\}\cup \{0\}.$\\
		Furthermore, if $\psi(a)=0$ then $\sigma(C_{\psi,\varphi};\mathcal{D}_{\alpha})=\{0\}.$
	\end{theorem}
	
	The paper is structured as follows. In Section \ref{S0}, we recall definitions and present preliminary results on weighted Dirichlet spaces. Section \ref{S1} is devoted to establishing necessary and sufficient conditions for the boundedness and compactness of $C_{\psi,\varphi}$, which depend on the interaction between $\psi$ and $\varphi$. Finally, in Section \ref{S2}, we examine the spectrum of $C_{\psi,\varphi}$ for specific choices of $\psi$ and $\varphi$.

	\section{Preliminaries}\label{S0}
	
	\noindent 
	Weighted Dirichlet spaces $\mathcal D_{\alpha}$ can be viewed as weighted Hardy spaces $H^2(\beta)$ with weight sequence $\beta(n)=(n+1)^{\frac{1-\alpha}{2}}.$ In particular, $\mathcal{D}_{1}$ coincides with the classical Hardy-Hilbert space $H^2(\mathbb{D}),$ $\mathcal{D}_{2}$ is the classical Bergman space $A^2(\mathbb{D})$ and $\mathcal{D}_{0}$ coincides with the classical Dirichlet space $\mathcal{D}.$ Each $\mathcal{D}_{\alpha}$ is a separable Hilbert space with inner product 
	\[\langle f,g \rangle_{\mathcal{D}_{\alpha}}=\sum_{n=0}^{\infty}(n+1)^{1-\alpha}a_n\overline{b}_n,\]
	where $f(z)=\sum\limits_{n=0}^{\infty}a_nz^n$ and $g(z)=\sum\limits_{n=0}^{\infty}b_nz^n.$ Moreover, these spaces are reproducing kernel Hilbert spaces, with reproducing kernel at $w \in \mathbb D$ given by
	\begin{align}\label{repker}
		k^{\alpha}_w(z)=
		\sum_{n=0}^{\infty}\frac{(\overline{w}z)^n}{(n+1)^{1-\alpha}},~~z \in \mathbb D.
	\end{align}
	The normalized Lebesgue area measure on $\mathbb{D}$ is denoted by $dA,$ and for $\alpha>-1$ the finite measure $dA_{\alpha}$ on $\mathbb{D}$ is given by 
	\[dA_{\alpha}(z)=(1-|z|^2)^{\alpha}dA(z).\]
	By using Stirling's formula for $f \in \mathcal{D}_{\alpha}$ it follows that 
	\begin{align}
		\|f\|^2_{\mathcal{D}_{\alpha}} \cong 
		\begin{cases}
			|f(0)|^2+\int_{\mathbb{D}}|f'(z)|^2dA_{\alpha}(z)\,\,\text{if}\,\, \alpha \in (-1,1),\\
			\int_{\mathbb{D}}|f(z)|^2dA_{\alpha-2}(z)\,\,\text{if}\,\,\alpha>1.
		\end{cases}
	\end{align}
	Throughout this paper, we write $A \lesssim B$ to mean $A \leq CB$ for an implicit constant $C>0,$ and $A \cong B$ means both $A \lesssim B$ and $B \lesssim A$ hold.\\
	Recall that for $\alpha>-1$ the weighted Bergman spaces $A_{\alpha}^2=A_{\alpha}^2(\mathbb{D})$ are defined by
	\[A^2_{\alpha}=\left\{f \in \text{Hol}(\mathbb{D}) : \int_{\mathbb{D}}|f(z)|^2dA_{\alpha}(z)<\infty \right\}.\]
	Again Stirling's formula shows that if $f \in A^2_{\alpha}$ with $f(z)=\sum_{n=0}^{\infty}a_nz^n$ then 
	\begin{align*}
		\|f\|^2_{A^2_{\alpha}}\cong \sum\limits_{n=0}^{\infty}(n+1)^{-1-\alpha}|a_n|^2.
	\end{align*}
	We therefore arrive at the following relations, which will play an important role in our subsequent proofs.
	\begin{align*}
		&(i)~~\mathcal{D}_{\alpha+2}=A^2_{\alpha}~~\text{for all}~~\alpha>-1.\\
		&(ii)~~\text{Let $f \in Hol(\mathbb D)$ and $\alpha>-1$. Then}~~ f \in \mathcal{D}_{\alpha}~~\text {if and only if}~~f' \in A^2_{\alpha}.\\
		&(iii)~~\text{The continuous inclusion}~~\mathcal{D}_{\alpha} \subseteq A^2_{\alpha}~~\text{holds for all}~~\alpha>-1, i.e., \|f\|_{A^2_{\alpha}} \lesssim \|f\|_{\mathcal{D}_{\alpha}}\\
		&~~\text{for all}~~f \in \mathcal{D}_{\alpha}.\\
		&(iv)~~\|f\|^2_{A^2_{\alpha}}\cong |f(0)|^2+\|f'\|^2_{A^2_{\alpha+2}}~~\text{for all}~~\alpha>-1.\\
		&(v)~~\|f\|^2_{\mathcal{D}_{\alpha}} \cong |f(0)|^2+|f'(0)|^2+\|f''\|^2_{A^2_{\alpha+2}}~~\text{for all}~~f \in \mathcal{D}_{\alpha}~~\text{and}~~\alpha \in (-1,1).
	\end{align*}
	For further details on weighted Dirichlet spaces, we refer to \cite{Cowen_BOOK, MS_CJM_1986}.

	\section{Boundedness and compactness}\label{S1}
	
	We first establish a necessary condition for the boundedness and compactness of $C_{\psi,\varphi}$ on $\mathcal D_{\alpha}$ for $\alpha \in (0,1).$
	
	
	

	\begin{proposition}\label{TB0}
		$(i)$~~If $C_{\psi,\varphi }$ is bounded on $\mathcal{D}_{\alpha}$ then 
		$$\sup_{z \in \mathbb{D}}|\psi(z)|\left(\frac{1-|z|^2}{1-|\varphi(z)|^2}\right)^{\frac{\alpha}{2}}<\infty.$$
		$(ii)$~~If $C_{\psi,\varphi }$ is compact on $\mathcal{D}_{\alpha}$ then 
		$$\lim_{|z|\to 1^-}|\psi(z)|\left(\frac{1-|z|^2}{1-|\varphi(z)|^2}\right)^{\frac{\alpha}{2}}=0.$$
	\end{proposition}
	\begin{proof}
		Let $\hat{k}^{\alpha}_z$ be the normalized reproducing kernel of $\mathcal{D}_{\alpha}$ at $z \in \mathbb{D}.$
		For all $\alpha>0,$ we have
		\begin{align*}
			\sum_{n=0}^{\infty}(n+1)^{\alpha-1}|z|^{2n}\cong (1-|z|^2)^{-\alpha}.
		\end{align*}
		This follows from the binomial series expansion
		\begin{align*}
			(1-|z|^2)^{-\alpha}=\sum_{n=0}^{\infty}\frac{\Gamma(n+\alpha)}{\Gamma(n+1)\Gamma(\alpha)}|z|^{2n},
		\end{align*}
		together with Stirling's formula, which gives
		\begin{align*}
			\frac{\Gamma(n+\alpha)}{\Gamma(n+1)}\cong (n+1)^{\alpha-1}.
		\end{align*}  
		Hence, for $\alpha \in (0,1),$ we have
		\begin{align*}
			\|k^{\alpha}_z\|_{\mathcal{D}_{\alpha}}\cong (1-|z|^2)^{-\frac{\alpha}{2}}~~\text{for all $z \in \mathbb{D}.$} 
		\end{align*}
		$(i)$~~Suppose that $C_{\psi,\varphi }$ is bounded on $\mathcal{D}_{\alpha}.$ Then, a direct consequence of the reproducing property of the kernel functions gives $C_{\psi,\varphi }^*k^{\alpha}_z=\overline{\psi(z)}k^{\alpha}_{\varphi(z)}.$
		It follows that
		\begin{align}\label{TC0E1}
			\|C_{\psi,\varphi }^*\hat{k}^{\alpha}_z\|_{\mathcal{D}_{\alpha}}=\frac{|\overline{\psi(z)}|\|k^{\alpha}_{\varphi(z)}\|_{\mathcal{D}_{\alpha}}}{\|k^{\alpha}_{z}\|_{\mathcal{D}_{\alpha}}} \cong |\psi(z)|\left(\frac{1-|z|^2}{1-|\varphi(z)|^2}\right)^{\frac{\alpha}{2}}.
		\end{align}
		Boundedness of $C_{\psi,\varphi }$ on $\mathcal{D}_{\alpha}$ implies that
		\begin{align*}
			\sup_{z \in \mathbb{D}}\|C_{\psi,\varphi }^*\hat{k}^{\alpha}_z\|_{\mathcal{D}_{\alpha}} \cong \sup_{z \in \mathbb{D}}|\psi(z)|\left(\frac{1-|z|^2}{1-|\varphi(z)|^2}\right)^{\frac{\alpha}{2}}<\infty,
		\end{align*}
		as desired.\\
		$(ii)$~~Suppose that $C_{\psi,\varphi }$ is compact on $\mathcal D_{\alpha}.$
		For $\alpha \in (0,1),$ it follows from \cite[Th. 2.17]{Cowen_BOOK} that the normalized kernels $\{\hat{k}^{\alpha}_z\}$ converges to  $0$ weakly as $|z|\to 1^-.$ 
		Since $C_{\psi,\varphi }$ is compact, so is its adjoint $C_{\psi,\varphi }^*,$ and hence $$\|C_{\psi,\varphi }^*\hat{k}^{\alpha}_z\|_{\mathcal{D}_{\alpha}}\to 0~~\text{as}~~|z|\to 1^{-}.$$
		Therefore, taking the limit $|z|\to 1^-$ on both sides of \eqref{TC0E1} yields the desired result.
	\end{proof}
	
	As an immediate consequence of Proposition \ref{TB0}, we obtain the following two corollaries.
	We first recall that $\text{Aut}(\mathbb{D})$ denotes the group of all bi-holomorphic self-maps of $\mathbb{D}$. From elementary complex analysis, each $\varphi \in \text{Aut}(\mathbb{D})$ can be expressed in the form
	$$\varphi(z)=\eta \frac{a-z}{1-\bar a z}, ~~\eta \in \partial \mathbb D,~~ a \in \mathbb D.$$
	Throughout this article, we use the notation $\varphi_a$ to denote the automorphism given by $\varphi_a(z)=\frac{a-z}{1-\bar a z}.$

	\begin{cor}\label{C0C1}
		For $\alpha \in (0,1),$ there exists no nonzero compact weighted composition operator $C_{\psi,\varphi }$ on $\mathcal{D}_{\alpha}$ with $\varphi \in \text{Aut}(\mathbb{D}).$ 
	\end{cor}
	
	\begin{proof}
		Let $\varphi \in \text{Aut}(\mathbb{D}).$ Then a simple computation shows that
		$$(1-|z|^2)|\varphi'(z)|=1-|\varphi(z)|^2~~\text{and}~~|\varphi'(z)|=|\varphi_a'(z)|\leq \frac{1+|a|}{1-|a|}.$$
		This implies that
		\begin{align}\label{COC1E1}
			|\psi(z)|\left(\frac{1-|a|}{1+|a|}\right)^{\frac{\alpha}{2}}	\leq |\psi(z)|\left(\frac{1-|z|^2}{1-|\varphi(z)|^2}\right)^{\frac{\alpha}{2}}.
		\end{align}
		Since $C_{\psi,\varphi }$ is compact on $\mathcal{D}_{\alpha},$ it follows from Proposition~\ref{TB0} that the expression on the right of inequality \eqref{COC1E1} tends to $0$ as $|z| \to 1^-.$ This implies 
		$$\lim_{|z|\to 1^-}	|\psi(z)|=0.$$
		As $\psi \in \text{Hol}(\mathbb{D}),$ the  Maximum Modulus Theorem ensures that $\psi\equiv 0.$
	\end{proof}

	\begin{cor}\label{C0C2}
		Let $C_{\psi,\varphi }$ be a compact operator on $\mathcal{D}_{\alpha}$ for some $\alpha \in (0,1).$ If $\varphi$ has no fixed point in $\mathbb{D}$ and $\psi$ is continuous on $\overline{\mathbb{D}}$ then $\psi$ necessarily vanishes at some point on $\partial \mathbb{D}.$  
	\end{cor}

	\begin{proof}
		Suppose that $C_{\psi,\varphi }$ is compact on $\mathcal{D}_{\alpha}$ for some $\alpha \in (0,1),$ and that $\varphi$ has no fixed points in $\mathbb{D}.$ Then there exits a point $\zeta \in \partial \mathbb{D}$ and a sequence $\{z_n\} \subset \mathbb{D}$ such that $$\lim_{n\to \infty}z_n=\lim_{n\to \infty}\varphi(z_n)=\zeta~~\text{and}~~\lim_{n\to \infty} \frac{1-|\varphi(z_n)|}{1-|z_n|}\leq 1,$$ 
		see \cite[p. 56]{Cowen_BOOK}. 
		From the Schwarz-Pick Theorem \cite[p. 48]{Cowen_BOOK} we have 
		$$\frac{1-|\varphi(z_n)|}{1-|z_n|}\geq \frac{1-|\varphi(0)|}{1+|\varphi(0)|}>0,$$
		and we get 
		$$\lim_{n\to \infty} \frac{1-|z_n|}{1-|\varphi(z_n)|}\geq 1.$$
		As $\psi$ is continuous on $\overline{\mathbb{D}},$ for each $\alpha \in (0,1),$ we obtain
		\begin{align}\label{ee11}
			|\psi(\zeta)|=\lim_{n\to \infty}|\psi(z_n)| \lesssim \lim_{n\to \infty}|\psi(z_n)|\left(\frac{1-|z_n|^2}{1-|\varphi(z_n)|^2}\right)^{\frac{\alpha}{2}}.
		\end{align}
		As $C_{\psi,\varphi }$ is compact on $\mathcal{D}_{\alpha},$ Proposition~\ref{TB0} implies that the expression on the right side of \eqref{ee11} tends to $0$ as $n \to \infty.$ Consequently, $\psi(\zeta)=0,$ which completes the proof.
	\end{proof}
	
	\begin{remark}
		Corollary~\ref{C0C2} provides a simple criterion to identify weighted composition operators that are not compact on $\mathcal D_{\alpha}$ for all $\alpha \in (0,1)$. For example, let $\varphi(z)=\frac{1+z^2}{2}$ and $\psi(z)=2+z,$ then $\varphi$ has no fixed point on $\mathbb D$ and $\psi$ has no zero on $\partial \mathbb D,$ it follows from Corollary~\ref{C0C2} that $C_{\psi,\varphi }$ is not compact on $\mathcal{D}_{\alpha}$ for all $\alpha \in (0,1).$
	\end{remark}

	Now consider the functions $\psi(z)=(1+z)^{1+\alpha}$ and $\varphi(z)=z^2.$ A simple computation yields 
	$$\sup_{z \in \mathbb{D}}|\psi(z)|\left(\frac{1-|z|^2}{1-|\varphi(z)|^2}\right)^{\frac{\alpha}{2}} \leq 2^{1+\alpha}.$$
	However, since $\psi \notin \mathcal D_{\alpha},$ the operator $C_{\psi,\varphi}$ is not bounded on $D_{\alpha}.$ This demonstrates that the condition in Proposition~\ref{TB0} is not sufficient for the boundedness of weighted composition operators.
	Therefore, extra conditions on the inducing functions are necessary to establish a sufficient condition for the boundedness of $C_{\psi,\varphi}.$
	
	We are now in a position to prove Theorem \ref{TB1}.
	
	\begin{proof}[Proof of Theorem \ref{TB1}]
		Let $f \in \mathcal{D}_{\alpha}.$ Then we have
		\begin{align}\label{TB1EE1}
			&\|C_{\psi,\varphi}f\|^2_{\mathcal{D}_{\alpha}}\nonumber\\
			&\cong |C_{\psi,\varphi}f(0)|^2+|(C_{\psi,\varphi}f)'(0)|^2+\|(C_{\psi,\varphi}f)''\|^2_{A^2_{\alpha+2}}\nonumber\\
			&= |\psi(0)f(\varphi(0))|^2+|(\psi f\circ \varphi)'(0)|^2+\|(\psi f\circ \varphi)''\|^2_{A^2_{\alpha+2}}\nonumber\\
			&\lesssim |\psi(0)f(\varphi(0))|^2+|\psi'(0)f(\varphi(0))|^2+|\psi(0)f'(\varphi(0))\varphi'(0)|^2+\|(\psi f\circ \varphi)''\|^2_{A^2_{\alpha+2}}.
		\end{align}
		Since $C_{\varphi }$ is bounded on $A^2_{\alpha}$ (see \cite{ZHU_BOOK}), we obtain
		\begin{align*}
			|f(\varphi(0))| \lesssim \|C_{\varphi }f\|_{A^2_{\alpha}} \lesssim \|f\|_{A^2_{\alpha}} \lesssim \|f\|_{D_{\alpha}}
		\end{align*}
		and 
		\begin{align*}
			|f'(\varphi(0))| \lesssim \|C_{\varphi }f'\|_{A^2_{\alpha}} \lesssim \|f'\|_{A^2_{\alpha}} \lesssim \|f\|_{D_{\alpha}}.
		\end{align*}
		This implies that
		\begin{align*}
			|\psi(0)f(\varphi(0))|^2+|\psi'(0)f(\varphi(0))|^2+|\psi(0)f'(\varphi(0))\varphi'(0)|^2 \lesssim \|f\|^2_{D_{\alpha}}.
		\end{align*}
		Now, 
		\begin{align*}
			&\|(\psi f\circ \varphi)''\|^2_{A^2_{\alpha+2}}\\
			&=\int_{\mathbb{D}}|(\psi(z) f(\varphi(z)))''|^2dA_{\alpha+2}(z)\\
			&\lesssim \int_{\mathbb{D}}|\psi''(z) f(\varphi(z))|^2dA_{\alpha+2}(z)+\int_{\mathbb{D}}|\psi'(z) f'(\varphi(z))\varphi'(z)|^2dA_{\alpha+2}(z)\\
			& +\int_{\mathbb{D}}|\psi(z) f'(\varphi(z))\varphi''(z)|^2dA_{\alpha+2}(z)+\int_{\mathbb{D}}|\psi(z) f''(\varphi(z))\varphi'^2(z)|^2dA_{\alpha+2}(z).
		\end{align*}
		By the first integral above and condition $(i)$ of the theorem, we obtain
		\begin{align*}
			\int_{\mathbb{D}}|\psi''(z) f(\varphi(z))|^2dA_{\alpha+2}(z) \lesssim \|C_{\varphi }f\|^2_{A^2_{\alpha}}\lesssim \|f\|^2_{A^2_{\alpha}} \lesssim \|f\|^2_{D_{\alpha}}.
		\end{align*}
		The second integral, together with condition $(ii),$ implies that
		\begin{align*}
			\int_{\mathbb{D}}|\psi'(z) f'(\varphi(z))\varphi'(z)|^2dA_{\alpha+2}(z) \lesssim \|C_{\varphi }f'\|^2_{A^2_{\alpha}}\lesssim \|f'\|^2_{A^2_{\alpha}} \lesssim \|f\|^2_{D_{\alpha}}.
		\end{align*}
		From the third integral, together with condition $(iii),$ we get
		\begin{align*}
			\int_{\mathbb{D}}|\psi(z) f'(\varphi(z))\varphi''(z)|^2dA_{\alpha+2}(z) \lesssim \|C_{\varphi }f'\|^2_{A^2_{\alpha}}\lesssim \|f'\|^2_{A^2_{\alpha}} \lesssim \|f\|^2_{D_{\alpha}}.
		\end{align*}
		Applying condition $(iv)$ to the fourth integral yields
		\begin{align*}
			\int_{\mathbb{D}}|\psi(z) f''(\varphi(z))\varphi'^2(z)|^2dA_{\alpha+2}(z) \lesssim \int_{\mathbb{D}}|f''(\varphi(z))|^2|\varphi'(z)|^2(1-|\varphi(z)|^2)^{\alpha+2}dA(z).
		\end{align*}
		As $\varphi$ is univalent on $\mathbb{D},$ a change of variable $w=\varphi(z)$ gives
		\begin{align*}
			&\int_{\mathbb{D}}|\psi(z) f''(\varphi(z))\varphi'^2(z)|^2dA_{\alpha+2}(z)\nonumber\\
			& \lesssim \int_{\mathbb{D}}|f''(w)|^2(1-|w|^2)^{\alpha+2}dA(w) =  \|f''\|^2_{A^2_{\alpha+2}}\lesssim \|f\|^2_{D_{\alpha}}.
		\end{align*}
		From all of the above relations, it follows that
		\begin{align*}
			\|C_{\psi,\varphi}f\|_{\mathcal{D}_{\alpha}} \lesssim \|f\|_{D_{\alpha}},
		\end{align*}
		as required.
	\end{proof}
	
	The linear fractional transformations (LFTs) on $\mathbb C$ are mappings of the form 
	$$z \to \frac{az+b}{cz+d}, ~~\text{where $a,b,c,d \in \mathbb C$ with $ad-bc\neq 0.$}$$
	In particular, if a LFT is a self-map of $\mathbb D$ then $|cz+d|$ is bounded away from zero on $\mathbb D$. Therefore it follows from Theorem \ref{TB1} that the composition operator induced by a LFT that is a self-map of $\mathbb D$ is bounded on $\mathcal D_{\alpha}$ for every $\alpha \in (-1,1).$ 
	Moreover, if both $\varphi$ and $\psi$ are LFTs that are also self-maps of $\mathbb D$, then Theorem~\ref{TB1} implies that the weighted composition operator $C_{\psi,\varphi }$ is bounded on $\mathcal D_{\alpha},$
	illustrating the applicability of Theorem~\ref{TB1} beyond the unweighted case.
	Furthermore, Theorem \ref{TB1} is useful for identifying functions that are not LFTs but nevertheless induce bounded weighted composition operators.
	We now give an illustrative example.
	
	\begin{example}
		For $r \in (0,1),$ let 
		$$\varphi_{r,1}(z)=exp\left(\frac{(1-r)(z+1)}{rz-1}\right)~~\text{and}~~\psi \equiv 1.$$ Then $\varphi_{r,1}$ is a univalent holomorphic self-map of $\mathbb D$ with $\varphi_{r,1}'' \in \text{H}^{\infty}(\mathbb D).$ A straightforward computation shows that
		$$\left(\frac{1-|z|^2}{1-|\varphi_{r,1}(z)|^2}\right)^{\frac{\alpha}{2}+1} \lesssim \left(\frac{1+r}{1-r}\right)^{\frac{\alpha}{2}+1}.$$
		Consequently, for each $r \in (0,1),$ the inducing functions $\varphi_{r,1}$ and $\psi$ satisfy the hypotheses of Theorem~\ref{TB1}, and hence the weighted composition operator $C_{\psi,\varphi_{r,1}}$ is bounded on $\mathcal D_{\alpha}$ for all $\alpha \in (-1,1).$ We note that the same conclusion also holds when $\psi$ is a polynomial, $\psi=\varphi_{r},$ $\psi=\varphi_{r,1}$ or $\psi(z)=(1-z)^{1+r}.$
	\end{example}

	We now establish a sufficient condition for $C_{\psi,\varphi }$ to be compact on $\mathcal D_{\alpha}$ for every $\alpha \in (-1,1).$ For this purpose, we require the following lemmas from the literature.

	\begin{lemma}\cite[Lemma 2.9]{Kumar_2009}\label{L1}
		Let $C_{\psi,\varphi }$ be a bounded operator on $\mathcal{D}_{\alpha}$ for $\alpha \in (-1,1).$ Then $C_{\psi,\varphi }$ is compact on $\mathcal{D}_{\alpha}$ if and only if whenever $\{f_n\}$ is a bounded sequence in $\mathcal{D}_{\alpha}$ converging to zero uniformly on compact subsets of $\mathbb{D},$ then $\|C_{\psi,\varphi }f_n\|_{\mathcal{D}_{\alpha}} \to 0.$
	\end{lemma}

	\begin{lemma}\cite[Th. 10.28]{RUDIN_BOOK}\label{L2}
		Suppose $f_n \in \text{Hol}(\mathbb{D}),$ for $n=1,2,\ldots,$ and $\{f_n\}$ converges to $f$ uniformly on compact subsets of $\mathbb{D}.$ Then $f \in \text{Hol}(\mathbb{D}),$ and $\{f'_n\}$ converges to $f'$ uniformly on compact subsets of $\mathbb{D}.$
	\end{lemma}

	\begin{proof}[Proof of Theorem \ref{TC1}]
		Let $\{f_n\}$ be a bounded sequence in $\mathcal{D}_{\alpha}$ converging to zero uniformly on compact subsets of $\mathbb{D}.$ It follows from Theorem \ref{TB1} that $C_{\psi,\varphi }$ is a bounded linear operator on $\mathcal{D}_{\alpha}.$ By Lemma~\ref{L1}, it suffices to show that $\|C_{\psi,\varphi }f_n\|_{\mathcal{D}_{\alpha}} \to 0.$ From \eqref{TB1EE1}, we have
		\begin{align*}
			&\|C_{\psi,\varphi}f_n\|^2_{\mathcal{D}_{\alpha}}\\
			&\lesssim |\psi(0)f_n(\varphi(0))|^2+|\psi'(0)f_n(\varphi(0))|^2+|\psi(0)f'_n(\varphi(0))\varphi'(0)|^2+\|(\psi f_n\circ \varphi)''\|^2_{A^2_{\alpha+2}}.
		\end{align*}
		Fix any $\epsilon>0.$ Since $\{\varphi(0)\}$ is a compact subset of $\mathbb{D},$ it follows from Lemma \ref{L2} that there exists $N_0 \in \mathbb{N}$ such that 
		\begin{align*}
			|\psi(0)f_n(\varphi(0))|^2+|\psi'(0)f_n(\varphi(0))|^2+|\psi(0)f'_n(\varphi(0))\varphi'(0)|^2 <\epsilon^2 ~~\forall n\geq N_0.
		\end{align*}
		Now, we have
		\begin{align*}
			&\|(\psi f_n\circ \varphi)''\|^2_{A^2_{\alpha+2}}\\
			&=\int_{\mathbb{D}}|(\psi(z) f_n(\varphi(z)))''|^2dA_{\alpha+2}(z)\\
			&\lesssim \int_{\mathbb{D}}|\psi''(z) f_n(\varphi(z))|^2dA_{\alpha+2}(z)+\int_{\mathbb{D}}|\psi'(z) f'_n(\varphi(z))\varphi'(z)|^2dA_{\alpha+2}(z)\\
			& +\int_{\mathbb{D}}|\psi(z) f'_n(\varphi(z))\varphi''(z)|^2dA_{\alpha+2}(z)+\int_{\mathbb{D}}|\psi(z) f''_n(\varphi(z))\varphi'^2(z)|^2dA_{\alpha+2}(z).
		\end{align*}
		From condition $(i),$ there exists $r_1 \in (0,1)$ such that 
		$$|\psi''(z)|(1-|z|^2)<\epsilon~~\text{for all}~~z \in \mathbb{D}\setminus \overline{B(0,r_1)}.$$
		Hence, we obtain
		\begin{align}\label{T1I11}
			&\int_{\mathbb{D}\setminus \overline{B(0,r_1)}}|\psi''(z) f_n(\varphi(z))|^2dA_{\alpha+2}(z)\nonumber\\
			& \lesssim \epsilon^2 \int_{\mathbb{D}}|f_n(\varphi(z))|^2dA_{\alpha}(z)
			=\epsilon^2 \|C_{\varphi}f_n\|^2_{A^2_{\alpha}} \lesssim \epsilon^2 \|f_n\|^2_{A^2_{\alpha}}\lesssim \epsilon^2 \|f_n\|^2_{D_{\alpha}} \lesssim \epsilon^2.
		\end{align}
		Continuity of $\varphi$ and $\psi''$ on $\overline{B(0,r_1)}$ implies that $\varphi(\overline{B(0,r_1)})$ is compact and  $\psi''$ is bounded on $\overline{B(0,r_1)}.$ By Lemma \ref{L2}, there exists $N_1\in \mathbb{N}$ such that 
		\[|f_n(\varphi(z))|<\epsilon~~\forall n \geq N_1,~~z \in \overline{B(0,r_1)}.\] 
		Therefore, for all $n \geq N_1,$ we have
		\begin{align}\label{T1I12}
			\int_{\overline{B(0,r_1)}}|\psi''(z) f_n(\varphi(z))|^2dA_{\alpha+2}(z) \lesssim \epsilon^2.
		\end{align}
		Combining \eqref{T1I11} and \eqref{T1I12}, we obtain, for all $n \geq N_1,$
		\begin{align*}
			\int_{\mathbb{D}}|\psi''(z) f_n(\varphi(z))|^2dA_{\alpha+2}(z) \lesssim \epsilon^2.
		\end{align*}
		By condition $(ii),$ there exists $r_2 \in (0,1)$ such that 
		$$|\varphi'(z)\psi'(z)|(1-|z|^2)<\epsilon~~\text{for all}~~z \in \mathbb{D}\setminus \overline{B(0,r_2)}.$$
		Then it follows that
		\begin{align}\label{T1I21}
			&\int_{\mathbb{D}\setminus \overline{B(0,r_2)}}|\psi'(z) f'_n(\varphi(z))\varphi'(z)|^2dA_{\alpha+2}(z)\nonumber\\
			& \lesssim \epsilon^2 \int_{\mathbb{D}}|f'_n(\varphi(z))|^2dA_{\alpha}(z)
			=\epsilon^2 \|C_{\varphi}f'_n\|^2_{A^2_{\alpha}} \lesssim \epsilon^2 \|f'_n\|^2_{A^2_{\alpha}}\lesssim \epsilon^2 \|f_n\|^2_{D_{\alpha}} \lesssim \epsilon^2.
		\end{align}
		Since $\varphi,\varphi'$ and $\psi'$ are continuous on $\overline{B(0,r_2)},$ it follows that $\varphi(\overline{B(0,r_2)})$ is compact and that $\varphi',\psi'$ are bounded on $\overline{B(0,r_2)}.$ By Lemma \ref{L2} there exists $N_2\in \mathbb{N}$ such that 
		\[|f'_n(\varphi(z))|<\epsilon~~\forall n \geq N_2,~~z \in \overline{B(0,r_2)}.\] 
		Thus, for all $n \geq N_2,$ we have
		\begin{align}\label{T1I22}
			\int_{\overline{B(0,r_2)}}|\psi'(z) f'_n(\varphi(z))\varphi'(z)|^2dA_{\alpha+2}(z) \lesssim \epsilon^2.
		\end{align}
		Combining \eqref{T1I21} and \eqref{T1I22}, for all $n \geq N_2,$ we get
		\begin{align*}
			\int_{\mathbb{D}}|\psi'(z) f'_n(\varphi(z))\varphi'(z)|^2dA_{\alpha+2}(z) \lesssim \epsilon^2.
		\end{align*}
		Condition $(iii)$ implies that there exists $r_3 \in (0,1)$ such that 
		$$|\varphi''(z)\psi(z)|(1-|z|^2)<\epsilon~~\text{for all}~~z \in \mathbb{D}\setminus \overline{B(0,r_3)}.$$
		Consequently, we have
		\begin{align}\label{T1I31}
			&\int_{\mathbb{D}\setminus \overline{B(0,r_3)}}|\psi(z) f'_n(\varphi(z))\varphi''(z)|^2dA_{\alpha+2}(z)\nonumber\\
			& \lesssim \epsilon^2 \int_{\mathbb{D}}|f'_n(\varphi(z))|^2dA_{\alpha}(z)
			=\epsilon^2 \|C_{\varphi}f'_n\|^2_{A^2_{\alpha}} \lesssim \epsilon^2 \|f'_n\|^2_{A^2_{\alpha}}\lesssim \epsilon^2 \|f_n\|^2_{D_{\alpha}} \lesssim \epsilon^2.
		\end{align}
		Continuity of $\varphi,\psi$ and $\varphi''$ on $\overline{B(0,r_3)}$ implies that $\varphi(\overline{B(0,r_3)})$ is compact and that $\psi$ and $\varphi''$ are bounded on $\overline{B(0,r_3)}.$ By Lemma \ref{L2}, there exists $N_3\in \mathbb{N}$ such that 
		\[|f'_n(\varphi(z))|<\epsilon~~\forall n \geq N_3,~~z \in \overline{B(0,r_3)}.\] 
		Therefore, for all $n \geq N_3,$ we have
		\begin{align}\label{T1I32}
			\int_{\overline{B(0,r_1)}}|\psi(z) f'_n(\varphi(z))\varphi''(z)|^2dA_{\alpha+2}(z) \lesssim \epsilon^2.
		\end{align}
		Combining \eqref{T1I31} and \eqref{T1I32}, for all $n \geq N_3$ we obtain
		\begin{align*}
			\int_{\mathbb{D}}|\psi(z) f'_n(\varphi(z))\varphi''(z)|^2dA_{\alpha+2}(z) \lesssim \epsilon^2.
		\end{align*}
		By condition $(iv),$ there exists $r_4 \in (0,1),$ such that 
		$$|\varphi'(z)\psi(z)|^2|(1-|z|^2)^{\alpha+2}<\epsilon^2 (1-|\varphi(z)|^2)^{\alpha+2} ~~\text{for all}~~z \in \mathbb{D}\setminus \overline{B(0,r_4)}.$$ 
		Hence, we obtain
		\begin{align*}
			&\int_{\mathbb{D}\setminus \overline{B(0,r_4)}}|\psi(z) f''_n(\varphi(z))\varphi'^2(z)|^2dA_{\alpha+2}(z)\\
			& \lesssim \epsilon^2 \int_{\mathbb{D}}|f''_n(\varphi(z))|^2|\varphi'(z)|^2(1-|\varphi(z)|^2)^{\alpha+2}dA(z).
		\end{align*}
		As $\varphi$ is univalent on $\mathbb{D},$ changing variables via $w=\varphi(z)$ yields
		\begin{align}\label{T1I41}
			&\int_{\mathbb{D}\setminus \overline{B(0,r_4)}}|\psi(z) f''_n(\varphi(z))\varphi'^2(z)|^2dA_{\alpha+2}(z)\nonumber\\
			&\lesssim \epsilon^2 \int_{\mathbb{D}}|f''_n(w)|^2(1-|w|^2)^{\alpha+2}dA(w)
			= \epsilon^2 \|f''_n\|^2_{A^2_{\alpha+2}}\lesssim \epsilon^2 \|f_n\|^2_{D_{\alpha}} \lesssim \epsilon^2.
		\end{align}
		Since $\varphi,\psi$ and $\varphi'$ are continuous on $\overline{B(0,r_4)},$ the set $\varphi(\overline{B(0,r_4)})$ is compact, and both $\psi$ and $\varphi'$ are bounded on $\overline{B(0,r_4)}.$ Applying Lemma~\ref{L2}, we can choose $N_4\in \mathbb{N}$ such that 
		\[|f''_n(\varphi(z))|<\epsilon~~\forall n \geq N_4,~~z \in \overline{B(0,r_4)}.\] 
		Consequently, for every $n \geq N_4,$ it follows that
		\begin{align}\label{T1I42}
			\int_{\overline{B(0,r_1)}}|\psi(z) f''_n(\varphi(z))\varphi'^2(z)|^2dA_{\alpha+2}(z) \lesssim \epsilon^2.
		\end{align}
		Combining \eqref{T1I41} and \eqref{T1I42}, for all $n \geq N_4,$ we get
		\begin{align*}
			\int_{\mathbb{D}}|\psi(z) f''_n(\varphi(z))\varphi'^2(z)|^2dA_{\alpha+2}(z) \lesssim \epsilon^2.
		\end{align*}
		Now, let 
		$$N=\max\{N_0,N_1,N_2,N_3,N_4\}.$$ 
		Finally, combining all the preceding relations, we obtain
		\begin{align*}
			\|C_{\psi,\varphi}f_n\|_{\mathcal{D}_{\alpha}} \lesssim \epsilon~~\forall n \geq N.
		\end{align*}
		This completes the proof.
	\end{proof}
	
	Theorem~\ref{TC1} offers a constructive method for generating compact weighted composition operators on $\mathcal D_{\alpha}$. As a specific example, if $\varphi$ is a LFT self-map of $\mathbb{D}$ satisfying $\|\varphi\|_{\infty}<1$ then it induces a compact composition operator $C_{\varphi}$ on $\mathcal D_{\alpha}$. This is consistent with the known result \cite[Th. 2]{Z_IUMJ_1990} characterizing compactness for $\alpha<0$. Furthermore, when $\psi$ is also an LFT self-map of $\mathbb D$, the weighted operator $C_{\psi,\varphi}$ is compact on $\mathcal D_{\alpha}$ for $\alpha \in (-1,1)$. This demonstrates that the applicability of Theorem~\ref{TC1} extends beyond the unweighted composition operators.
	
	In the following example, we exhibit a self-map $\varphi$ of $\mathbb D$ with $\|\varphi\|_{\infty}=1$ for which the weighted composition operator $C_{\psi,\varphi}$ is compact on $\mathcal D_{\alpha}$ for every $\alpha \in (-1,1),$ even though neither the composition operator $C_{\varphi}$ nor the multiplication operator $M_{\psi}$ is compact on $\mathcal D_{\alpha}$ for all $\alpha \in (0,1).$

	\begin{example}\label{EX1}
		Let 
		$\psi(z)=(1-z)^{2+\alpha}~~\text{and}~~\varphi(z)=\frac{\lambda_{\varphi}z}{1-(1-\lambda_{\varphi})z},$ where $\lambda_{\varphi}\in [\frac{1}{2},1).$
		Then $\varphi$ is a univalent holomorphic self-map of $\mathbb{D}$ with fixed points $0$ and it admits angular derivatives at every point of $\partial \mathbb D.$
		By \cite[Th. 5.3]{MS_CJM_1986}, the composition operator $C_{\varphi}$ is not compact on $\mathcal D_{\alpha}$ for any $\alpha \in (-1,1).$ Moreover, by Corollary~\ref{C0C1}, the multiplication operator $M_{\psi}$ is not compact on $\mathcal{D}_{\alpha}$ for every $\alpha \in (0,1).$ 
		
		A direct computation shows that $\psi$ and $\varphi$ satisfy conditions $(ii)$ and $(iii)$ of Theorem \ref{TC1}.
		Furthermore, we have
		\begin{align*}
			|\psi''(z)|\lesssim \begin{cases}
				1\,\,\,\,\text{if}~~\alpha \in [0,1)\\
				(1-|z|)^{\alpha}\,\,\,\,\text{if}~~\alpha \in (-1,0),
			\end{cases}
		\end{align*}
		which implies that condition $(i)$ of Theorem \ref{TC1} is satisfied for all
		$\alpha \in (-1,1).$
		Finally, for $\lambda_{\varphi}\in [\frac{1}{2},1),$ a straightforward computation yields
		\begin{align*}
			|\psi(z)|\left(\frac{1-|z|^2}{1-|\varphi(z)|^2}\right)^{\frac{\alpha}{2}+1} \lesssim (1-|z|^2)^{\frac{\alpha}{2}+1},
		\end{align*}
		which verifies condition $(iv)$ of Theorem \ref{TC1}.
		Consequently, $C_{\psi,\varphi }$ is compact on $\mathcal{D}_{\alpha}$ for every $\alpha \in (-1,1).$
	\end{example}

	This example points to an interesting observation that a non-compact composition operator on $\mathcal D_{\alpha}$ can be made compact by choosing an appropriate weight function. Furthermore, in Example \ref{EX1} the weight function vanishes on $\varphi(\partial \mathbb D)\cap \partial \mathbb D.$ This observation leads to the following question:
	\begin{question}
		Let $C_{\psi,\varphi}$ be bounded operator on $\mathcal D_{\alpha}.$
		Suppose that $\varphi$ is continuous on $\partial \mathbb D$ and $\psi$ is continuous and vanishes on $\varphi(\partial \mathbb D)\cap \partial \mathbb D.$ Is  $C_{\psi,\varphi}$ necessarily compact on $\mathcal D_{\alpha}?$
	\end{question}

	The following lemma will be used in proving the next theorem.
	
	
	\begin{lemma}\label{Englis_p3}
		If $f \in \text{Hol}(\mathbb{D})$ and 
		$\sup\limits_{z \in \mathbb{D}}|f'(z)|(1-|z|^2)<\infty,$
		then 
		$$\lim_{|z|\to 1^-}|f(z)|(1-|z|^2)^{\alpha}=0~~\text{for every $\alpha>0$}.$$
	\end{lemma}
	
	\begin{proof}
		For $f \in \text{Hol}(\mathbb{D})$ and $z \in \mathbb{D},$ we get
		\begin{align*}
			|f(z)-f(0)|=\left|\int_{[0,z]}f'(\xi)d\xi \right|\leq {M} \int^{|z|}_{0}\frac{d\xi}{1-\xi^2}\leq M\left(\log2+ \frac{1}{2}\log\left(\frac{1}{1-|z|^2}\right) \right),
		\end{align*}
		where $M=\sup\limits_{z \in \mathbb{D}}(1-|z|^2)|f'(z)|.$
		Since, $M<\infty,$ it follows that for all $z \in \mathbb{D},$
		\begin{align}\label{eeeee}
			|f(z)|\lesssim 1+ \log\left(\frac{1}{1-|z|^2}\right).
		\end{align}
		Finally, multiplying both sides of \eqref{eeeee} by $(1-|z|^2)^{\alpha},$ and letting $|z| \to 1^-,$ we obtain the desired result.
	\end{proof}


	We now prove the characterization of bounded and compact weighted composition operators stated in the Introduction.
	
	\begin{proof}[Proof of Theorem \ref{mainp1}]
		One direction is immediate from Proposition~\ref{TB0}, so it remains to prove the converse.
		
		Since $\lim\limits_{|z|\to 1^-}|\psi''(z)|(1-|z|^2)=0$ and $\varphi''\in \text{H}^{\infty}(\mathbb{D}),$ Lemma~\ref{Englis_p3} implies that conditions $(i), (ii)$ and $(iii)$ of Theorem \ref{TC1} hold.
		It remains to show that 
		$$\lim_{|z|\to 1^-}|\psi(z)|\left(\frac{1-|z|^2}{1-|\varphi(z)|^2}\right)^{\frac{\alpha}{2}}=0$$ 
		implies 
		$$\lim_{|z|\to 1^-}|\psi(z)|\left(\frac{1-|z|^2}{1-|\varphi(z)|^2}\right)^{\frac{\alpha}{2}+1}=0.$$
		This follows from the Schwarz-Pick Theorem \cite[p. 48]{Cowen_BOOK}, which gives for any holomorphic self-map $\varphi$ of $\mathbb{D},$
		$\frac{1-|z|}{1-|\varphi(z)|}\leq \frac{1+|\varphi(0)|}{1-|\varphi(0)|}.$
		
		The boundedness case follows by an entirely parallel argument, replacing the limits as $|z|\to 1^-$ with supremum over $z \in \mathbb D$ and using Theorem \ref{TB1} instead of Theorem \ref{TC1}.
	\end{proof}
	

	

	In the next result, we extend the classical Comparison Theorem to weighted composition operators $C_{\psi,\varphi}$ for suitable inducing functions $\psi$ and $\varphi,$ providing a natural weighted analogue of the classical theorem. Recall that the Comparison Theorem \cite[Th. 5.2]{MS_CJM_1986} states that if $C_{\varphi}$ is bounded or compact on $\mathcal D_{\alpha}$ for some $\alpha \in (-1,1),$ then it has the same property on $\mathcal D_{\beta}$ for all $\beta \in (\alpha,1).$

	\begin{proposition}\label{mainp2}
		Let $\varphi$ be a univalent holomorphic self-map of $\mathbb{D}$ with $\varphi''\in \text{H}^{\infty}(\mathbb{D}),$ and let $\psi \in \text{Hol}(\mathbb D).$\\
		$(i)$~~If $\sup\limits_{z \in \mathbb{D}}|\psi''(z)|(1-|z|^2)<\infty$ and $C_{\psi,\varphi}$ is bounded on $\mathcal{D}_{\alpha}$ for some $\alpha \in (0,1)$ then $C_{\psi,\varphi}$ is bounded on $\mathcal{D}_{\beta}$ for every $\beta \in (\alpha,1).$\\
		$(ii)$ ~~If $\lim\limits_{|z|\to 1^-}|\psi''(z)|(1-|z|^2)=0$ and $C_{\psi,\varphi}$ is compact on $\mathcal{D}_{\alpha}$ for some $\alpha \in (0,1)$ then $C_{\psi,\varphi}$ is compact on $\mathcal{D}_{\beta}$ for every $\beta \in (\alpha,1).$
	\end{proposition}

	\begin{proof}
		The proof is divided into two steps: the case where $\varphi$ fixes the origin, followed by the general case.\\
		Step 1: Since $\varphi(0)=0,$ the Schwarz Lemma implies that  
		$$1-|z|^2\leq 1-|\varphi(z)|^2~~\text{for all $z \in \mathbb{D}$}$$ 
		and hence
		$$|\psi(z)|\left(\frac{1-|z|^2}{1-|\varphi(z)|^2}\right)^{\frac{\beta}{2}}\leq |\psi(z)|\left(\frac{1-|z|^2}{1-|\varphi(z)|^2}\right)^{\frac{\alpha}{2}}~~\text{for all $\beta \in (\alpha,1).$}$$
		The compactness condition follows by letting $|z| \to 1^-$ and applying Theorem~\ref{mainp1}~$(ii)$, while the boundedness condition follows by taking the supremum over $z \in \mathbb D$ and applying Theorem~\ref{mainp1}~$(i).$\\
		Step 2: Suppose that $\varphi$ is a holomorphic self-map of $\mathbb{D}$ that does not fix the origin. Then there exists $a \in \mathbb{D}\setminus \{0\}$ such that $\varphi(0)=a.$ The function $\varphi_a\circ \varphi$ then fixes the origin and is a univalent holomorphic self-map of $\mathbb{D}$ with $(\varphi_a\circ \varphi)''\in \text{H}^{\infty}(\mathbb{D}).$ A straightforward computation shows that
		\begin{align}\label{mainp2e1}
			\left(\frac{1-|z|^2}{1-|\varphi(z)|^2}\right)=|\varphi_a'(\varphi(z))|\left(\frac{1-|z|^2}{1-|\varphi_a(\varphi(z))|^2}\right)~~\forall z \in \mathbb{D}.
		\end{align}
		Since for all $z \in \mathbb{D},$
		$$\frac{1-|a|^2}{4} \leq |\varphi_a'(\varphi(z))|\leq \frac{1+|a|}{1-|a|}.$$
		Thus from the equality \eqref{mainp2e1} for any $\alpha \in (0,1),$ we obtain
		\begin{align}\label{nextsp}
			\left(\frac{1-|z|^2}{1-|\varphi(z)|^2}\right)^{\frac{\alpha}{2}} \cong 	\left(\frac{1-|z|^2}{1-|\varphi_a(\varphi(z))|^2}\right)^{\frac{\alpha}{2}}.
		\end{align}
		Now, letting $|z| \to 1^-$ on both sides of \eqref{nextsp} implies that
		\begin{align*}
			\lim_{|z|\to 1^-}|\psi(z)|\left(\frac{1-|z|^2}{1-|\varphi(z)|^2}\right)^{\frac{\alpha}{2}}=0~~\text{iff}~~\lim_{|z|\to 1^-}|\psi(z)|\left(\frac{1-|z|^2}{1-|\varphi_a(\varphi(z))|^2}\right)^{\frac{\alpha}{2}}=0.
		\end{align*}
		By Theorem \ref{mainp1}~$(ii)$, for each $\alpha \in (0,1),$ the operator $C_{\psi,\varphi}$ is compact on $\mathcal{D}_{\alpha}$ if and only if $C_{\psi,\varphi_a\circ \varphi}$ is compact on $\mathcal{D}_{\alpha}.$ Hence, the compactness of $C_{\psi,\varphi}$ on $\mathcal D_{\alpha}$ implies the compactness of $C_{\psi,\varphi_a\circ \varphi}.$
		By Step 1, $C_{\psi,\varphi_a\circ \varphi}$ is compact on $\mathcal{D}_{\beta}$ for every $\beta \in (\alpha,1).$ Consequently, $C_{\psi,\varphi}$ is compact on $\mathcal{D}_{\beta}$ for every $\beta \in (\alpha,1).$
		
		The boundedness follows from a parallel argument, replacing the limit as $|z|\to 1^-$ with the supremum over $z \in \mathbb D.$ 
	\end{proof}

	\section{Spectral properties}\label{S2}
	
	It is known by Schwarz-Pick theorem that any nontrivial holomorphic self-map of the unit disc has at most one fixed point in $\mathbb D.$ In this section, we investigate the spectrum of the weighted composition operators $C_{\psi, \varphi}$ under the additional assumption that $\varphi$ has a fixed point in $\mathbb D.$ 
	We begin by proving several preliminary results.
	
	In our first result, we assume that this fixed point is the origin and work with functions $\varphi$ of the form
	\begin{align}\label{al1}
		\varphi(z)=z\tau(z)~~\text{for all}~~ z \in \mathbb D,
	\end{align}
	where $\tau \in \text{Hol}(\mathbb D).$
	
	\begin{lemma}\label{SPL1}
		Let $\varphi$ be a holomorphic self-map of $\mathbb D$ in the form \eqref{al1}. Then $\varphi'' \in \text{H}^{\infty}(\mathbb D)$ if and only if $\tau'' \in \text{H}^{\infty}(\mathbb D).$
	\end{lemma}
	
	\begin{proof}
		For all $z \in \mathbb D,$ differentiating \eqref{al1} yields
		\begin{align}\label{al2}
			\varphi'(z)=z\tau'(z)+\tau(z)
		\end{align}
		and
		\begin{align}\label{al3}
			\varphi''(z)=z\tau''(z)+2\tau'(z).
		\end{align}
		Suppose that $\tau'' \in \text{H}^{\infty}(\mathbb D).$ Then it follows from \eqref{al3} that $\varphi'' \in \text{H}^{\infty}(\mathbb D).$ 
		
		Conversely, assume that $\varphi'' \in \text{H}^{\infty}(\mathbb D).$ Now, we fix $0<\delta<1.$ The continuity of $\tau$ on $\overline{B(0,\delta)}$ implies that $\tau$ is bounded on $\overline{B(0,\delta)}.$ From \eqref{al1}, we obtain
		$$|\tau(z)|=\frac{|\varphi(z)|}{|z|} \leq \frac{|\varphi(z)|}{\delta}\,\,\,\,\text{for all}\,\,\,\, z \in \mathbb D \setminus \overline{B(0,\delta)}.$$
		Since $\varphi \in \text{H}^{\infty}(\mathbb D),$ it follows that $\tau \in \text{H}^{\infty}(\mathbb D).$ In a similar way, \eqref{al2} together with $\tau, \varphi' \in \text{H}^{\infty}(\mathbb D)$ yields $\tau' \in \text{H}^{\infty}(\mathbb D)$. Similarly, \eqref{al3} along with $\tau', \varphi'' \in \text{H}^{\infty}(\mathbb D)$ implies $\tau'' \in \text{H}^{\infty}(\mathbb D)$.     
	\end{proof}
	
	To establish the next result, we recall the following. For $\alpha \in (-1,1)$, $\mathcal D_{\alpha}$ possesses an orthonormal basis $\mathcal{B} = \{e_n\}_{n=0}^\infty$, where $e_n(z) = (n+1)^{\frac{\alpha-1}{2}} z^n.$ Furthermore, each function $e_n$ is a multiplier on $\mathcal{D}_{\alpha}$, i.e., $e_n \mathcal{D}_{\alpha} \subseteq \mathcal{D}_{\alpha}.$ A comprehensive treatment of multipliers on $\mathcal{D}_{\alpha}$ spaces can be found in \cite{T_TAMS_1966}.

	\begin{lemma}\label{SPL2}
		Let $\varphi$ be a univalent holomorphic self-map of $\mathbb{D}$ satisfying 
		$\varphi'' \in \text{H}^{\infty}(\mathbb D)$ and $\varphi(0)=0.$ 
		For a given $\alpha \in (-1,1),$ let $\psi \in \text{Hol}(\mathbb{D})$ be such that:
		\begin{align*}
			&(i)~~\sup_{z \in \mathbb{D}}|\psi''(z)|(1-|z|^2)<\infty,\\
			&(ii)~~\sup_{z \in \mathbb{D}}|\psi(z)|\left(\frac{1-|z|^2}{1-|\varphi(z)|^2}\right)^{\frac{\alpha}{2}+1}<\infty.
		\end{align*}
		Then $C_{\psi,\varphi}(e_n\mathcal{D}_{\alpha}) \subseteq e_n\mathcal{D}_{\alpha}$ for all $n \in \mathbb N \cup \{0\}.$ 
	\end{lemma}
	
	\begin{proof}
		Since $\varphi(0)=0$, there exists $\tau \in \text{Hol}(\mathbb D)$ such that $\varphi(z)=z\tau(z),$ for all $z \in \mathbb D.$
		For any $f \in \mathcal{D}_{\alpha},$ we have
		\begin{align*}
			C_{\psi,\varphi}(e_nf)(z)=(n+1)^{\frac{\alpha-1}{2}}z^n(\tau(z))^n\psi(z)f(\varphi(z))=(n+1)^{\frac{\alpha-1}{2}}z^n\zeta(z)\psi(z)f(\varphi(z)),
		\end{align*}
		where $\zeta(z)=(\tau(z))^n.$ To prove the desired result, it suffices to show that $\zeta\psi f\circ\varphi \in \mathcal{D}_{\alpha},$ which is equivalent to $(\zeta\psi f\circ\varphi)'' \in A^2_{\alpha+2}.$ We compute
		\begin{align}\label{ll}
			(\zeta\psi f\circ\varphi)''=\zeta''(\psi f\circ\varphi)+2\zeta'(\psi f\circ\varphi)'+\zeta(\psi f\circ\varphi)''.
		\end{align}
		Proceeding as in the proof of Theorem~\ref{TB1}, we get $(\psi f\circ\varphi)'' \in A^2_{\alpha+2}.$ Consequently, $(\psi f\circ\varphi)' \in A^2_{\alpha}$ and $\psi f\circ\varphi \in \mathcal{D}_{\alpha}.$ Since $\mathcal{D}_{\alpha} \subseteq A^2_{\alpha} \subseteq A^2_{\alpha+2},$ it follows that 
		$$\psi f\circ\varphi, (\psi f\circ\varphi)', (\psi f\circ\varphi)'' \in A^2_{\alpha+2}.$$
		By Lemma~\ref{SPL1}, we have
		$\tau'' \in \text{H}^{\infty}(\mathbb D),$ which implies $\zeta, \zeta', \zeta'' \in \text{H}^{\infty}(\mathbb D).$ Therefore, from \eqref{ll} we conclude that $(\zeta\psi f\circ\varphi)'' \in A^2_{\alpha+2}.$ This completes the proof.
	\end{proof}
	
	This leads to the following result concerning the spectrum.
	
	\begin{proposition}\label{SPT1}
		Let $\varphi$ be a univalent holomorphic self-map of $\mathbb{D}$ with fixed point $a \in \mathbb D$ and satisfies $\varphi'' \in \text{H}^{\infty}(\mathbb D)$. 
		For a given $\alpha \in (-1,1),$ let $\psi \in \text{Hol}(\mathbb{D})$ be such that:
		\begin{align*}
			&(i)~~\sup_{z \in \mathbb{D}}|\psi''(z)|(1-|z|^2)<\infty,\\
			&(ii)~~\sup_{z \in \mathbb{D}}|\psi(z)|\left(\frac{1-|z|^2}{1-|\varphi(z)|^2}\right)^{\frac{\alpha}{2}+1}<\infty.
		\end{align*}
		Then $0<|\varphi'(a)|\leq 1$ and $\sigma(C_{\psi,\varphi};\mathcal{D}_{\alpha})\supseteq \left\{\psi(a)(\varphi'(a))^n: n \in \mathbb N \cup \{0\} \right\}.$
	\end{proposition}
	\begin{proof}
		Assume $\psi$ and $\varphi$ satisfy the hypotheses of the theorem. The univalence of $\varphi$ implies $|\varphi'(a)|>0$, while the Schwarz-Pick Theorem yields $|\varphi'(a)|\leq 1$.\\
		We prove this result by considering two steps. In Step 1, we establish the result for $\varphi$ fixing the origin, and in Step 2, we prove the general case. \\
		Step 1: Let $n \in \mathbb N,$ and let $M_n$ be the subspace of $\mathcal{D}_{\alpha}$ spanned by $\{e_0,e_1, \ldots, e_{n-1}\}.$ Then the weighted Dirichlet space $\mathcal{D}_{\alpha}$ can be decomposed as 
		$$\mathcal{D}_{\alpha}=M_n\oplus e_n\mathcal{D}_{\alpha}.$$
		By Lemma \ref{SPL2}, $C_{\psi,\varphi}(e_n\mathcal{D}_{\alpha}) \subseteq e_n\mathcal{D}_{\alpha},$ and hence $C^*_{\psi,\varphi}(M_n)\subseteq M_n.$ Consequently, the matrix representation of the compression $C^*_{\psi,\varphi}|_{M_n}$ with respect to the orthonormal basis $\mathcal B_n=\{e_0,e_1, \ldots, e_{n-1}\}$ is upper triangular. The point spectrum of $C^*_{\psi,\varphi}|_{M_n}$ consists of the diagonal entries $\langle C^*_{\psi,\varphi}e_k, e_k\rangle_{\mathcal{D}_{\alpha}}$ for $k=0,1,\ldots, n-1,$ and the corresponding eigenvectors lie in $M_n.$
		We now compute the diagonal entries of the matrix. For $k=0,$
		$\langle C^*_{\psi,\varphi}e_0, e_0\rangle_{\mathcal{D}_{\alpha}}=\overline{\psi(0)}.$
		For $k \in \mathbb N,$ we have
		\begin{align}\label{SPT1E1}
			\langle C^*_{\psi,\varphi}e_k, e_k\rangle_{\mathcal{D}_{\alpha}}=\langle e_k, C_{\psi,\varphi}e_k\rangle_{\mathcal{D}_{\alpha}}=(n+1)^{\frac{1-\alpha}{2}}\langle e_k, \psi \varphi^k\rangle_{\mathcal{D}_{\alpha}}.
		\end{align}
		Since $\varphi(0)=0,$ the power series expansion of $\varphi$ is
		$$\varphi(z)=\varphi'(0)z+\text{higher order terms in $z$}.$$
		Hence,
		$$\psi(z)(\varphi(z))^k=\psi(0)(\varphi'(0))^kz^k+\text{higher order terms in $z$}.$$
		Therefore, from \eqref{SPT1E1}, we obtain
		$$\langle C^*_{\psi,\varphi}e_k, e_k\rangle_{\mathcal{D}_{\alpha}}=\overline{\psi(0)} (\overline{\varphi'(0)})^k.$$
		Since the eigenvalues of a compression of an operator are contained in the eigenvalues of the original operator, it follows that
		$$\sigma(C^*_{\psi,\varphi};\mathcal{D}_{\alpha})\supseteq \left\{\overline{\psi(0)} (\overline{\varphi'(0)})^n: n \in \mathbb N \cup \{0\} \right\}.$$
		By the spectral mapping theorem, the spectrum of an operator and its adjoint are related as
		$$\sigma(C^*_{\psi,\varphi};\mathcal{D}_{\alpha})=\left\{\overline{\lambda} : \lambda \in \sigma(C_{\psi,\varphi};\mathcal{D}_{\alpha}) \right\}.$$
		Therefore, we obtain 
		$$\sigma(C_{\psi,\varphi};\mathcal{D}_{\alpha})\supseteq \left\{\psi(0)(\varphi'(0))^n: n \in \mathbb N \cup \{0\} \right\}.$$\\
		Step 2: We now consider the general case. Let $\zeta=\psi \circ \varphi_a$ and $\eta=\varphi_a\circ \varphi \circ \varphi_a.$ Then $\eta$ is a univalent holomorphic self-map of $\mathbb D$ with $\eta'' \in \text{H}^{\infty}(\mathbb D)$ and $\eta(0)=0.$\\
		It follows that
		\begin{align*}
			\sup_{z \in \mathbb D}|\zeta''(z)|(1-|z|^2)
			\lesssim & \sup_{z \in \mathbb D}\left\{ \left(|\psi''(\varphi_a(z))|+|\psi'(\varphi_a(z))|\right)(1-|z|^2) \right\}\\
			=&\sup_{z \in \mathbb D}\left\{ \left(|\psi''(z)|+|\psi'(z)|\right)(1-|\varphi_a(z)|^2) \right\}\\
			=&\sup_{z \in \mathbb D}\left\{ \left(|\psi''(z)|+|\psi'(z)|\right)|\varphi'_a(z)|(1-|z|^2) \right\}\\
			\cong & \sup_{z \in \mathbb D}\left\{ \left(|\psi''(z)|+|\psi'(z)|\right)(1-|z|^2) \right\}.
		\end{align*}
		Since $\sup\limits_{z \in \mathbb{D}}|\psi''(z)|(1-|z|^2)<\infty,$ and by Lemma \ref{Englis_p3},
		$\sup\limits_{z \in \mathbb D} |\psi'(z)|(1-|z|^2)<\infty,$ it follows that $\sup\limits_{z \in \mathbb D}|\zeta''(z)|(1-|z|^2)<\infty.$\\
		Additionally,
		\begin{align*}
			\sup_{z \in \mathbb{D}}|\zeta(z)|\left(\frac{1-|z|^2}{1-|\eta(z)|^2}\right)^{\frac{\alpha}{2}+1}&=\sup_{z \in \mathbb{D}}|\psi(z)|\left(\frac{1-|\varphi_a(z)|^2}{1-|\varphi_a(\varphi(z))|^2}\right)^{\frac{\alpha}{2}+1}\\
			& \cong \sup_{z \in \mathbb{D}}|\psi(z)|\left(\frac{1-|z|^2}{1-|\varphi_a(\varphi(z))|^2}\right)^{\frac{\alpha}{2}+1}.
		\end{align*}
		Thus, by \eqref{nextsp}, we have 
		\begin{align*}
			\sup_{z \in \mathbb{D}}|\zeta(z)|\left(\frac{1-|z|^2}{1-|\eta(z)|^2}\right)^{\frac{\alpha}{2}+1} \cong \sup_{z \in \mathbb{D}}|\psi(z)|\left(\frac{1-|z|^2}{1-|\varphi(z)|^2}\right)^{\frac{\alpha}{2}+1}<\infty.
		\end{align*}
		Therefore, by Step 1, we obtain
		$$\sigma(C_{\zeta,\eta};\mathcal{D}_{\alpha})
		\supseteq \left\{\zeta(0)(\eta'(0))^n: n \in \mathbb N \cup \{0\} \right\}=\left\{\psi(a)(\varphi'(a))^n: n \in \mathbb N \cup \{0\} \right\}.$$
		The weighted composition operators $C_{\zeta,\eta}$ and $C_{\psi,\varphi}$ are similar, since
		$$C_{\zeta,\eta}=C^{-1}_{\varphi_a}C_{\psi,\varphi}C_{\varphi_a}.$$
		Hence, $C_{\zeta,\eta}$ and $C_{\psi,\varphi}$ have the same spectra, and we obtain
		\begin{align*}
			\sigma(C_{\psi,\varphi};\mathcal{D}_{\alpha})=\sigma(C_{\zeta,\eta};\mathcal{D}_{\alpha})\supseteq
			\left\{\psi(a)(\varphi'(a))^n: n \in \mathbb N \cup \{0\} \right\},
		\end{align*}
		as desired.
	\end{proof}

	The eigenfunctions of $C_{\varphi}$ are the solutions to the Schr\"{o}der equation (see \cite{S_MA_1871}):
	\begin{align}\label{SE}
		f \circ \varphi = \lambda f.
	\end{align}
	In \cite{K_ASS_1884}, the author studied this equation for holomorphic self-maps $\varphi$ of $\mathbb{D}$ that have a fixed point in $\mathbb{D}$, and found the values of $\lambda$ for which a solution exists.
	
	In the following lemma, we perform an analogous analysis for the weighted equation
	\begin{align}\label{WSE}
		\psi f \circ \varphi = \lambda f.
	\end{align}
	This can be seen as the weighted version of \eqref{SE}.

	\begin{lemma}\label{SPMSL1}
		Let $\varphi$ be a holomorphic self-map of $\mathbb D$ with fixed point $a \in \mathbb D$ and $\psi \in \text{Hol}(\mathbb D).$ If there exists $f \in \text{Hol}(\mathbb D)$ that satisfies the equation \eqref{WSE}.
		Then either there exits $n \in \mathbb N \cup \{0\}$ such that $\lambda=\psi(a)(\varphi'(a))^n$ or $f \equiv 0.$
	\end{lemma}
	
	\begin{proof}
		Suppose that $f \in \text{Hol}(\mathbb D)$ satisfies the equation \eqref{WSE}.
		Since $\varphi(a)=a$ for $a \in \mathbb D,$ it follows from \eqref{WSE} that 
		$(\psi(a)-\lambda) f(a)=0.$ Hence, either $\lambda=\varphi(a)$ or $f(a)=0.$ 
		If $\lambda \neq \varphi(a)$ then we have $f(a)=0.$ Differentiating both sides of \eqref{WSE} gives
		\begin{align*}
			\psi' (f\circ \varphi)+\psi (f'\circ \varphi)\varphi'=\lambda f'.
		\end{align*}
		Evaluating at $a$ yields that $(\psi(a)\varphi'(a)-\lambda) f'(a)=0,$ which implies either $\lambda=\psi(a)\varphi'(a)$ or $f(a)=f'(a)=0.$ 
		Now, differentiating both sides of \eqref{WSE} $k$-times, we get
		\begin{align*}
			\psi^{(k)} (f\circ \varphi)+\psi (f^{(k)}\circ \varphi)(\varphi')^k+\text{ terms containing $(f^{(m)}\circ \varphi)$ for $m < k$ }=\lambda f^{(k)}.
		\end{align*}
		If $f(a)=f^{(m)}(a)=0$ for $m < k,$ then we obtain $(\psi(a)(\varphi'(a))^k-\lambda) f^{(k)}(a)=0.$ This implies that either $\lambda=\psi(a)(\varphi'(a))^k$ or $f^{(k)}(a)=0.$ Consequently, if $\lambda\neq \psi(a)(\varphi'(a))^n$ for every $n \in \mathbb N \cup \{0\}$ then $f(a)=f^{(k)}(a)=0$ for all $k \in \mathbb N.$ Hence, $f \equiv 0,$ as desired.
		
	\end{proof}

	We now prove the main result of this section.
	
	\begin{proof}[Proof of Theorem \ref{SPMT1}]
		Suppose that $\psi$ and $\varphi$ satisfy the hypotheses of the theorem. By Proposition \ref{SPT1}, we have $0<|\varphi'(a)|\leq 1.$ If $|\varphi'(a)|= 1,$ the Schwarz Lemma implies $\varphi$ is an automorphism of $\mathbb D$. In that case, Corollary \ref{C0C1} gives $\psi \equiv 0$. Therefore, $0<|\varphi'(a)|< 1.$
		
		By the Schwarz-Pick Theorem, $\varphi$ has at most one fixed point in $\mathbb D$, otherwise, $\varphi$ is the identity map. In the latter case, our assumption $(ii)$ yields $\lim\limits_{|z|\to 1^-}|\psi(z)|=0,$ and by the Maximum Modulus Theorem this implies $\psi\equiv 0.$
		Now, suppose $\lambda$ is a non-zero element of $\sigma(C_{\psi,\varphi};\mathcal{D}_{\alpha})$ not of the form $\psi(a)(\varphi'(a))^n$. Since $a$ is the unique fixed point of $\varphi$ in $\mathbb D$, Lemma \ref{SPMSL1} implies that any corresponding eigenfunction $f$ must vanish identically. Therefore, $\lambda$ cannot be an eigenvalue.
		By Theorem \ref{TC1}, $C_{\psi,\varphi}$ is compact on $\mathcal{D}_{\alpha}$, and since every non-zero spectral value of a compact operator is an eigenvalue, we conclude via Proposition \ref{SPT1} that no such $\lambda$ exists. This completes the proof of the spectral characterization.     
	\end{proof}

	To illustrate this, we apply our main result to determine the spectrum for a number of specific weighted composition operators.

	\begin{example}\label{EXX}
		$(i)$~~ Consider $\psi(z)=(1-z)^{2+\alpha}$ and $\varphi(z)=\frac{\lambda_{\varphi}z}{1-(1-\lambda_{\varphi})z}$ with $\lambda_{\varphi}\in [\frac{1}{2},1).$
		Example \ref{EX1} shows that $\psi$ and $\varphi$ meet the assumptions of Theorem \ref{SPMT1}. Therefore, the spectrum of $C_{\psi,\varphi}$ on $\mathcal D_{\alpha}$ is given by
		$$\sigma(C_{\psi,\varphi};\mathcal{D}_{\alpha})=\left\{\lambda_{\varphi}^n: n \in \mathbb N \cup \{0\} \right\}\cup \{0\}.$$


		\noindent 
		$(ii)$ For $r \in (0,1)$ and $k>1,$ let 
		$$\varphi_{r,k}(z)=exp\left(\frac{z(rk-1)+(r-k)}{1-rz}\right)~~\text{and}~~\psi(z)= z^2.$$ 
		By the Denjoy-Wolff Theorem (see \cite{Cowen_BOOK}), $\varphi_{r,k}$ has a fixed point in $\mathbb D,$ say $a_{r,k}$.
		The function $\varphi_{r,k}$ is a univalent holomorphic self-map on $\mathbb D$ and satisfies $\varphi_{r,k}'' \in \text{H}^{\infty}(\mathbb D).$ A straightforward computation shows that
		$$\left(\frac{1-|z|^2}{1-|\varphi_{r,k}(z)|^2}\right)^{\frac{\alpha}{2}+1} \lesssim (1-|z|^2)^{\frac{\alpha}{2}+1}.$$
		Consequently, for each $r \in (0,1)$ and $k>1,$ the inducing functions $\varphi_{r,k}$ and $\psi$ satisfy the hypotheses of Theorem \ref{SPMT1}. 
		Therefore, the spectrum of $C_{\psi,\varphi_{r,k}}$ on $\mathcal D_{\alpha}$ is given by
		$$\sigma(C_{\psi,\varphi};\mathcal{D}_{\alpha})=\left\{a^2_{r,k}(\varphi'_{r,k}(a_{r,k}))^n: n \in \mathbb N \cup \{0\} \right\}\cup \{0\}.$$
	\end{example}
	
	We conclude by posing the following question, which concerns a possible generalization of Theorem \ref{SPMT1}.
	
	\begin{question}
		Does there exist a non-zero compact weighted composition operator $C_{\psi,\varphi}$ on $\mathcal{D}_{\alpha}$ for which $\varphi$ has no fixed point in $\mathbb{D}$? If so, can the spectrum of $C_{\psi,\varphi}$ be characterized in this case?
	\end{question}


	\section*{Declarations}	
	\textit{Acknowledgements.} This research is supported by Czech Science Foundation (GA CR) grant no. 25-18042S. The author is grateful to Dr. Petr Blaschke for his careful reading of the manuscript and his valuable suggestions for improving the article. \\
	\textit{Data Availability :} No data was used for the research described in this article.\\
	\textit{Conflict of interest:} The author declares no conflict of interest.\\

	\bibliographystyle{amsplain}

\begin{thebibliography}{99}
		
		
		\bibitem{AL_IJM_2004} R. Aron and M. Lindstr\"{o}m, Spectra of weighted composition operators on weighted Banach spaces of analytic functions, Israel J. Math. 141 (2004), 263--276.
		
		
		\bibitem{CS_PAMS_1975} J. G. Caughran and H. J. Schwartz, Spectra of compact composition operators, Proc. Amer. Math. Soc. 51 (1975), 127--130.
		
		
		\bibitem{CGP_MA_2015} I. Chalendar, E. A. Gallardo-Gutiérrez and J. R. Partington, Weighted composition operators on the Dirichlet space: boundedness and spectral properties, Math. Ann. 363 (2015), no. 3-4, 1265--1279.
		
		
		
		\bibitem{Cowen_BOOK} C. C. Cowen and  B. D. MacCluer, Composition operators on spaces of analytic functions, Stud. Adv. Math. CRC Press, Boca Raton, FL, (1995).
		
		
		
		
		\bibitem{CR_JLMS_2004} \v{Z}.\ \v{C}u\v{c}kovi\'{c} and R. Zhao, Weighted composition operators on the Bergman space, J. London Math. Soc. (2) 70 (2004), no. 2, 499--511.
		
		
		
		
		
		\bibitem{GS_JFA_2016} E. A. Gallardo-Gutiérrez and R. Schroderus, The spectra of linear fractional composition operators on weighted Dirichlet spaces, J. Funct. Anal. 271 (2016), no. 3, 720--745.
		
		
		\bibitem{G_PAMS_2008} G. Gunatillake, Compact weighted composition operators on the Hardy space, Proc. Amer. Math. Soc. 135 (2007), no. 2, 461--467.
		
		
		
		\bibitem{H_ADM_1997} P. R. Hurst, Relating composition operators on different weighted Hardy spaces, Arch. Math. (Basel) 68 (1997), no. 6, 503--513.
		
		\bibitem{K_PJM_1979} H. Kamowitz, Compact operators of the form $uC_{\varphi},$ Pacific J. Math. 80 (1979), no. 1, 205--211.
		
		\bibitem{K_ASS_1884} G. K{\"o}nigs, Recherches sur les int{\'e}grales de certaines {\'e}quations fonctionnelles, Ann. Sci. {\'E}c. Norm. Sup{\'e}r. 1 (1884), no. 3, 3--41.
		
		\bibitem{Kumar_2009} S. Kumar, Weighted composition operators between spaces of Dirichlet type, Rev. Mat. Complut. 22 (2009), 469--488.
		
		\bibitem{LL_JAMS_2022} C. -O. Lo  and  A. W. -K. Loh, Compact and Hilbert-Schmidt weighted composition operators on weighted Bergman spaces, J. Aust. Math. Soc. 113 (2022), no. 2, 208--225.
		
		
		\bibitem{MS_CJM_1986} B. D. MacCluer and J. H. Shapiro, Angular derivatives and compact composition operators on
		the Hardy and Bergman spaces, Canad. J. Math. 38 (1986), no. 4, 878--906.
		
		
		
		\bibitem{RUDIN_BOOK} W. Rudin, Real and complex analysis, McGraw-Hill Book Co., New York, (1987).
		
		
		\bibitem{S_MA_1871} E. Schr\"{o}der, \"{U}ber iterierte Funktionen, Math. Ann. 3 (1871), 296--322.
		
		\bibitem{Shaprio_BOOK} J. H. Shaprio, Composition operators and classical function theory, Universitext Tracts Math., Springer-Verlag, New York, (1993).
		
		
		
		
		\bibitem{T_TAMS_1966} G. D. Taylor, Multipliers on $D_{\alpha}$, Trans. Amer. Math. Soc. 123 (1966), 229--240.
		
		\bibitem{T_TAMS_2003} M. Tjani, Compact composition operators on Besov spaces,
		Trans. Amer. Math. Soc. 355 (2003), no. 11, 4683--4698.
		
		
		
		
		\bibitem{ZHU_BOOK}  K. Zhu, Operator theory in function spaces, Math. Surveys Monogr., 138
		American Mathematical Society, Providence, RI, (2007).
		
		
		
		\bibitem{Z_PAMS_1998}  N. Zorboska, Composition operators on weighted Dirichlet spaces, Proc. Amer. Math. Soc. 126 (1998), no. 7, 2013--2023.
		
		\bibitem{Z_IUMJ_1990} N. Zorboska, Composition operators on $S_a$ spaces, Indiana Univ. Math. J. 39 (1990), no. 3, 847--857.
		
		
		
		
		
		
		
		
		
		
		
		
		
		
		
		
		
		
		
		
		
		
	\end{thebibliography}

\end{document}